\documentclass{amsart}

\usepackage[T1]{fontenc}
\usepackage{times}
\usepackage{amssymb,epsfig,verbatim,xypic}
\theoremstyle{plain}

\usepackage[T1]{fontenc}
\usepackage[utf8]{inputenc}
\usepackage[english]{babel}
\usepackage{amssymb}
\usepackage{amsthm}
\usepackage{graphics}
\usepackage{amsmath}
\usepackage{amstext}
\usepackage{multirow}
\usepackage{comment}

\usepackage{graphicx}
\usepackage{amsmath}
\usepackage{amssymb}
\usepackage{amsthm}
\usepackage{amstext}
\usepackage{epstopdf}
\usepackage{mathrsfs}
\usepackage{amscd}
\usepackage{xypic}

\usepackage[all,cmtip]{xy}

\newtheorem{thm}{Theorem}[section]
\newtheorem{dfn}[thm]{Definition}
\newtheorem{lemma}[thm]{Lemma}
\newtheorem{prop}[thm]{Proposition}
\newtheorem{cor}[thm]{Corollary}

\newtheorem*{THM}{Main Theorem}
\newtheorem*{PROP}{Proposition}

\theoremstyle{remark}
\newtheorem{ex}[thm]{Example}
\newtheorem{rem}[thm]{Remark}

\newcommand{\mb}{\mathbb}

\newcommand{\R}{\mb R}
\newcommand{\C}{\mb C}
\newcommand{\Pj}{\mb P}
\newcommand{\Z}{\mb Z}
\newcommand{\Q}{\mb Q}
\newcommand{\N}{\mb N}

\newcommand{\rato}{\dashrightarrow}

\DeclareMathOperator{\codim}{codim}
\DeclareMathOperator{\Sym}{Sym}
\DeclareMathOperator{\Proj}{Proj}
\DeclareMathOperator{\Aut}{Aut}
\DeclareMathOperator{\Bir}{Bir}

\numberwithin{equation}{section}

\addtocounter{section}{0}             
\numberwithin{equation}{section}       

\title{On the primitivity of birational transformations of irreducible symplectic manifolds}
\author{Federico Lo Bianco}
\date{\today}

\begin{document}
\maketitle

\begin{abstract}
Let $f\colon X\rato X$ be a bimeromorphic transformation of a complex irreducible symplectic manifold $X$. Some important dynamical properties of $f$ are encoded by the induced linear automorphism $f^*$ of $H^2(X,\Z)$.
Our main result is that a bimeromorphic transformation such that $f^*$ has at least one eigenvalue with modulus $>1$ doesn't admit any invariant fibration (in particular its generic orbit is Zariski-dense).
\end{abstract}

\section{Introduction}
\label{intro}

A complex manifold is said \textbf{irreducible symplectic} if it is simply connected and the vector space of holomorphic $2$-forms is spanned by a nowhere degenerate form. Irreducible symplectic manifolds form, together with Calabi-Yau manifolds and complex tori, one of the three fundamental classes of K\"ahler manifolds with trivial canonical bundle. We are going to denote by $X$ an irreducible symplectic manifold and by $f\colon X\rato X$ a bimeromorphic transformation of $X$.\\
On the second cohomology of $X$ we can define a quadratic form, the Beauville-Bogomolov form, whose restriction to $H^{1,1}(X,\R)$ is hyperbolic (i.e. has signature $(1, h^{1,1}(X)-1)$) and which is preserved by the linear pull-back action $f^*$ induced by $f$ on cohomology; the setting is therefore similar to that of a compact complex surface, where the intersection form makes the second cohomology group into a hyperbolic lattice. In the surface case, the action of an automorphism $f\colon S\to S$ on cohomology translates into dynamical properties of $f$ (see Paragraphs \ref{isometries hyperbolic} and \ref{parabolic case} for details), and we can hope to have similar results in the irreducible symplectic case.

If $g\colon M\rato M$ is a meromorphic transformation of a compact K\"ahler manifold $M$, for $p=0,1,\ldots ,\dim (M)$ the $p$-th \textbf{dynamical degree} of $g$ is
$$\lambda_p(g):= \limsup_{n\to +\infty}||(g^n)_p^*||^{\frac 1n},$$
where $(g^n)_p^*\colon H^{p,p}(M)\to H^{p,p}(M)$ is the linear morphism induced by $g^n$ and $||\cdot ||$ is any norm on the space $End (H^{p,p}(M))$. Note that in the case of an automorphism, $\lambda_p(f)$ is just the maximal modulus of eigenvalues of $f^*_p$.

Let $g\colon M\rato M$ be a bimeromorphic transformation of a compact K\"ahler manifold. A meromorphic fibration $\pi\colon M\rato B$ onto a compact K\"ahler manifold $B$ such that $\dim B\neq 0, \dim X$ is called \textbf{$g$-invariant} if there exists a bimeromorphic transformation $h\colon B\rato B$ such that $\pi\circ g=h\circ \pi$.\\
\centerline{\xymatrix{
M \ar@{-->}[d]^\pi \ar@{-->}[r]^g & M \ar@{-->}[d]^\pi\\
B \ar@{-->}[r]^h & B
}}
The transformation $g$ is said to be \textbf{primitive} (see \cite{MR3431659}) if it admits no invariant fibration.

In the surface case, an automorphism whose action on cohomology has infinite order admits an invariant fibration onto a curve if and only if all the dynamical degrees are equal to $1$ (Theorem \ref{fibrations surfaces}). Our main result establish an analogue of the "only if" direction.


\begin{THM}
\label{thm A}
Let $X$ be an irreducible symplectic manifold, $f\colon X\rato X$ a bimeromorphic transformation with at least one dynamical degree $>1$. Then
\begin{enumerate}
\item $f$ is primitive;
\item $f$ admits at most $\dim(X)+b_2(X)-2$ periodic hypersurfaces;
\item the generic orbit of $f$ is Zariski-dense.
\end{enumerate}
\end{THM}

Here a hypersurface $H\subset X$ is said to be $f$-periodic if its strict transform $(f^n)^*H$ by some iterate of $f$ is equal to $H$.

\begin{rem}
Point $(2)$ follows from point $(1)$ and \cite[Theorem B]{MR2727612}; point $(3)$ follows from point $(1)$ and \cite[Theorem 4.1]{MR2400885}, but is proven here as a lemma (Lemma \ref{dense orbits}).
\end{rem}


In order to prove the Main Theorem, we establish a result on the dynamics of birational transformations of projective manifolds that has its own interest.


\begin{PROP}
\label{thm B}
Let $X,B$ be projective manifolds, $f\colon X\rato X, g\colon B\rato B$ birational transformations and $\pi\colon X\to B$ a non-trivial fibration such that $\pi\circ f=g\circ \pi$. If the generic orbit of $g$ is Zariski-dense and the generic fibre of $\pi$ is of general type, then
\begin{enumerate}
\item $\pi$ is isotrivial over an open dense subset $U\subset B$;
\item there exists an étale cover $U'\to U$ such that the induced fibration $X'=U'\times _U \pi^{-1}(U)$ is trivial: $X'\cong U'\times F$ for a fibre $F$;
\item the images by the natural morphism $X'\to \pi^{-1}(U)$ of the varieties $U'\times \{ pt\}$ are $f$-periodic; in particular the generic orbit of $f$ is not Zariski-dense.
\end{enumerate}
\end{PROP}
\begin{rem}
Point $(1)$ is equivalent to point $(2)$ by \cite[Proposition 2.6.10]{MR2247603}.
\end{rem}

In Section \ref{dyn deg} we recall the definition and main results about dynamical degrees, in the absolute and relative context; Section \ref{sect HK} is consecrated to irreducible symplectic manifolds, with a focus on the invariance of the Beauville-Bogomolov form under the action of a birational transformation; in Section \ref{proof main results} and \ref{sec: invariant subvarieties} we prove the Main Theorem and the Proposition above; Section \ref{alternative approach} presents a different approach to the proof of the Main Theorem, which allows to prove a slightly weaker version of it.

\subsection{Acknowledgements}
I would like to thank Serge Cantat for proposing the topic and strategy of this paper, and for his precious help in all the stages of its redaction; I am also grateful to Mathieu Romagny for the fruitful conversations about Hilbert schemes.

\section{Dynamical degrees}
\label{dyn deg}

Throughout this section $M$ will be a compact K\"ahler manifold of dimension $d$.

\subsection{Definition and entropy}

Let $f\colon X\rato Y$ be a dominant meromorphic map between compact K\"ahler manifolds; the map $f$ is then holomorphic outside its indeterminacy locus $\mathcal I\subset X$, which has codimension at least $2$. The closure $\Gamma $ of its graph over $X\setminus \mathcal I$ is an irreducible analytic subset of dimension  $d$ in $X\times Y$. Let $\pi_X,\pi_Y$ denote the restrictions to $\Gamma$ of the projections from $X\times Y$ to $X$ and to $Y$ respectively; then $\pi_X$ induces a biholomorphism $\pi_X^{-1}(X\setminus \mathcal I)\cong X\setminus \mathcal I$ and we can identify $f$ with $\pi_Y\circ \pi_X^{-1}$.\\
Let $\alpha$ be a smooth $(p,q)$-form on $Y$; we define the pull-back of $\alpha$ by $f$ as the $(p,q)$-current (see \cite{demailly1997complex} for the basic theory of currents) on $X$
$$f^*\alpha := (\pi_X)_*(\pi_Y^*\alpha).$$
It is not difficult to see that if $\alpha$ is  closed (resp. positive), then so is $f^* \alpha$, so that $f$ induces a linear morphism between the Hodge cohomology groups. This definition of pull-back coincides with the usual one when $f$ is holomorphic.

Remember that the $p$-th dynamical degree of a dominant meromorphic map $f\colon M\rato M$ are defined as
$$\lambda_p(f)= \limsup_{n\to +\infty}||(f^n)_p^*||^{\frac 1n}.$$
Thanks to the above definition of pull-back, one can prove that
$$\lambda_p(f)=\lim_{n\to +\infty} \left( \int_M (f^n)^*\omega^p \wedge \omega^{d-p}\right)^{\frac 1n}$$
for any K\"ahler form $\omega$. See \cite{MR2119243}, \cite{MR2920277} for details.

The $p$-th dynamical degree measures the exponential growth of the volume of $f^n(V)$ for subvarieties $V\subset M$ of dimension $p$ \cite{MR2139697}.

\begin{rem}
By definition $\lambda_0(f)=1$; $\lambda_d(f)$ coincides with the topological degree of $f$: it is equal to the number of points in a generic fibre of $f$.
\end{rem}

\begin{rem}
Let $f$ be an automorphism. Then we have $(f^n)^*=(f^*)^n$, so that $\lambda_p(f)$ is the maximal modulus of eigenvalues of the linear automorphism $f^*_p\colon H^{p,p}(M,\R)\to H^{p,p}(M,\R)$; since $f^*$ also preserves the positive cone $\mathcal K_p\subset H^{p,p}(M,\R)$, a theorem of Birkhoff \cite{MR0214605} implies that $\lambda_p(f)$ is a positive real eigenvalue of $f^*_p$.\\
It should be noted however that in the bimeromorphic setting we have in general $(f^n)^*\neq (f^*)^n$.
\end{rem}

\begin{rem}
\label{inverse degrees}
If $f$ is bimeromorphic we have
$$\lambda_p(f)=\lambda_{d-p}(f^{-1}).$$
Indeed, for $f$ biregular we have
$$\int_M (f^n)^*\omega^p \wedge \omega^{d-p}=\int_M (f^{-n})^*(f^n)^*\omega^p \wedge (f^{-n})^*\omega^{d-p}=\int_M \omega^p\wedge (f^{-n})^*\omega^{d-p},$$
which proves the equality by taking the limit.\\
If $f$ is only bimeromorphic, for all $n$ we can find two dense open subsets $U_n,V_n\subset M$ such that $f^n$ induces an isomorphism $U_n\cong V_n$; by the definition of pull-back the measures $(f^n)^*\omega^p \wedge \omega^{d-p}$ and $\omega^p\wedge (f^{-n})^*\omega^{d-p}$ have no mass on any proper closed analytic subset, so that
$$\int_M (f^n)^*\omega^p \wedge \omega^{d-p}=\int_{U_n} (f^n)^*\omega^p \wedge \omega^{d-p}=\int_{V_n} \omega^p\wedge (f^{-n})^*\omega^{d-p}=\int_M \omega^p\wedge (f^{-n})^*\omega^{d-p},$$
which proves the equality in the bimeromorphic case as well.
\end{rem}

The main interest in the definition of dynamical degrees lies in the following theorem by Yomdin and Gromov \cite{MR1095529}.

\begin{thm}
If $f\colon M\to M$ is an automorphism, then the topological entropy of $f$ is given by
$$h_{top}(f)=\max_{p=0,\ldots, d} \log\lambda_p(f).$$
\end{thm}

The topological entropy is a positive real number which measures the disorder created by iterations of $f$.\\
It is also possible to give a definition of topological entropy in the bimeromorphic context (see \cite{MR2391122}), but in this situation we only have
$$h_{top}(f)\leq \max_{p=0,\ldots, d} \log\lambda_p(f).$$

\subsection{Relative setting}

Dinh, Nguy\^en and Truong have studied the behaviour of dynamical degrees in the relative setting (\cite{MR2851870} and \cite{MR2989646}). Throughout this paragraph we denote by $f\colon M\rato M$ a meromorphic transformation of a compact K\"ahler manifold $M$ of dimension $d$, by $\pi\colon M\rato B$ a meromorphic fibration onto a compact K\"ahler manifold $B$ of dimension $k$ and by $g\colon B\rato B$ a meromorphic transformation such that
$$g\circ \pi=\pi\circ f.$$

The $p$-th \textbf{relative dynamical degree} of $f$ is defined as
$$\lambda_p(f|\pi)=\limsup_{n\to +\infty}\left(  \int_M (f^n)^*\omega_M^p\wedge \pi^*\omega_B^k \wedge \omega_M^{d-p-k}  \right)^{\frac 1n},$$
where $\omega_M$ and $\omega_B$ are arbitrary K\"ahler forms on $M$ and $B$ respectively. In particular $\lambda_p(f|\pi)=0$ for $p>d-k$.

Roughly speaking, $\lambda_p(f|\pi)$ gives the exponential growth of $(f^n)^*$ acting on the subspace of classes in $H^{p+k,p+k}(M,\R)$ that can be supported on a generic fibre of $\pi$; if $M$ is projective, it also represents the growth of the volume of $f^n(V)$ for $p$-dimensional subvarieties $V\subset \pi^{-1}(b)$ of a generic fibre of $\pi$.

\begin{rem}
\label{bimeromorphic invariant}
Dynamical degrees and relative dynamical degrees are bimeromorphic invariants \cite{MR2851870}. In other words, if there exist bimeromorphic maps $\phi\colon M\rato M'$, $\psi\colon B\rato B'$ and a meromorphic fibration $\pi'\colon M'\rato B'$ such that $\pi'\circ \phi=\psi\circ \pi$, then
$$\lambda_p(f)=\lambda_p(\phi\circ f\circ \phi^{-1}),\qquad \lambda_q(f|\pi)=\lambda_q(\phi\circ f\circ \phi^{-1}|\pi').$$
\end{rem}

\begin{rem}
\label{relative=fibre}
If $F=g^{-1}(b)$ is a regular, $f$-invariant, non-multiple fibre, then $\lambda_p(f|\pi)=\lambda_p(f_{|F})$ for all $p$ (see \cite{MR2851870}).
\end{rem}

The following theorem is due to Dinh, Nguy\^en and Truong \cite{MR2851870}.

\begin{thm}
\label{dinh nguyen}
Let $M$ be a compact K\"ahler manifold, $f\colon M\rato M$ a meromorphic transformation, $\pi\colon M\rato B$ a meromorphic fibration and $g\colon B\rato B$ a meromorphic transformation such that $\pi\circ f=g\circ \pi$. Then for all $p=0,\ldots \dim(M)$
$$\lambda_p(f)=\max_{q+r=p}\lambda_q(f|\pi)\lambda_r(g).$$
\end{thm}

\subsection{Log-concavity}

Dynamical degrees and their relative counterparts enjoy a log-concavity property (see \cite{khovanskii1979geometry},\cite{MR524795}, \cite{MR1095529}, \cite{MR2920277} for the original result, \cite{MR2851870} for the relative setting).

\begin{prop}
\label{log-concavity}
If $f\colon M\rato M$ is a meromorphic dominant map, the sequence $p\mapsto \log \lambda_p(f)$ is concave on the set $\{0,1,\ldots , d\}$; in other words
$$\lambda_p(f)^2\geq \lambda_{p-1}(f)\lambda_{p+1}(f) \qquad \text{for }p=1,\ldots, d-1.$$
Analogously, if $\pi\colon M\rato B$ is an $f$-invariant meromorphic fibration, then the sequence $p\mapsto \log \lambda_p(f|\pi)$ is concave on the set $\{0,1,\ldots , \dim(M)-\dim(B)\}$.
\end{prop}

As a consequence we have $\lambda_p\geq 1$ for all $p=0,\ldots, d$; furthermore, there exist $0\leq p\leq p+q\leq d$ such that
\begin{equation}
1=\lambda_0(f)<\cdots < \lambda_p(f)=\lambda_{p+1}(f)=\cdots=\lambda_{p+q}(f)>\cdots > \lambda_d(f).
\end{equation}

\section{Irreducible symplectic manifolds}
\label{sect HK}

We give here the basic notions and properties of irreducible symplectic manifolds (see \cite{MR1963559}, \cite{MR2964480} for details).

\begin{rem}
Because of the non-degeneracy of $\sigma$, one can easily prove that an irreducible symplectic manifold has even complex dimension.
\end{rem}

Throughout this section $X$ denotes an irreducible symplectic manifold of dimension $2n$ and $\sigma$ a non-degenerate holomorphic two-form on $X$.

Here is a list of the known examples of such manifolds that are not deformation equivalent.

\begin{enumerate}
\item Let $S$ be a $K3$ surface, i.e. a simply connected K\"ahler surface with trivial canonical bundle. Then the Hilbert scheme $S^{[n]}=Hilb^n(S)$, parametrizing $0$-dimensional subschemes of $S$ of length $n$, is a $2n$-dimensional irreducible symplectic manifold.
\item Let $T$ be a complex torus of dimension $2$, let $\phi\colon Hilb^n(T)\to \Sym^n(T)$ be the natural morphism and let $s\colon \Sym^n(T)\to T$ be the sum morphism. Then the kernel $K_{n-1}(T)$ of the composition $s\circ \phi$ is an irreducible symplectic manifold of dimension $2n-2$, which is called a \emph{generalized Kummer variety}.
\item O'Grady has found two sporadic examples of irreducible symplectic manifolds of dimension $6$ and $10$.
\end{enumerate}

An irreducible symplectic manifold is said of type $K3^{[n]}$ (respectively of type generalized Kummer) if it is deformation equivalent to $Hilb^n(S)$ for some $K3$ surface $S$ (respectively to $K_{n-1}(T)$ for some two-dimensional complex torus $T$).

\subsection{The Beauville-Bogomolov form}

We can define a natural quadratic form on the second cohomology $H^2(X,\R)$ which enjoys similar properties to the intersection form on compact surfaces; for details and proofs see \cite{MR1963559}.

\begin{dfn}
Let $\sigma$ be a holomorphic two-form such that $\int(\sigma\bar \sigma)^n=1$. The \emph{Beauville-Bogomolov quadratic form $q_{BB}$} on $H^2(X,\R)$ is defined by
$$q_{BB}(\alpha)=\frac n2 \int_X \alpha^2(\sigma\bar \sigma)^{n-1}+(1-n)\left(\int_X \alpha\sigma^n\bar\sigma^{n-1}\right)\left(\int_X \alpha\sigma^{n-1}\bar \sigma^n\right).$$
\end{dfn}

The Beauville-Bogomolov form satisfies two important properties: first the Beauville-Fujiki relation, saying that there exists a constant $c>0$ such that
$$q_{BB}(\alpha)^n=c\int_X \alpha^{2n} \qquad \text{for all }\alpha\in H^2(X,\R).$$
In particular, some multiple of $q_{BB}$ is defined over $\Z$.\\
Second, the next Proposition describes completely the signature of the form.

\begin{prop}
\label{signature BB}
The Beauville-Bogomolov form has signature $(3,b_2(X)-3)$ on $H^2(X,\R)$.\\
More precisely, the decomposition $H^2(X,\R)=H^{1,1}(X,\R)\oplus \left(H^{2,0}(X)\oplus H^{0,2}(X)\right)_\R$ is orthogonal with respect to $q_{BB}$, and $q_{BB}$ has signature $(1,h^{1,1}(X)-1)$ on $H^{1,1}(X,\R)$ and is positive definite on $\left(H^{2,0}(X)\oplus H^{0,2}(X)\right)_\R$.
\end{prop}

\begin{rem}
\label{non-fixed q non-negative}
For a divisor $D\in Div(X)$, we define $q_{BB}(D):=q_{BB}(c_1(\mathcal O_X(D)))$.\\
If $D$ is effective and without fixed components, then $q_{BB}(D)\geq 0$. Indeed, let $D'$ be an effective divisor linearly equivalent to $D$ and with no components in common with $D$. We have
$$q_{BB}(D)=\frac n2 \int_{D\cap D'}(\sigma \bar\sigma)^{n-1},$$
where each irreducible component of the intersection $D\cap D'$ is counted with its multiplicity. The integral on the right hand side is non-negative because $\sigma$ is a holomorphic form.\\
If furthermore $D$ is ample, then by Beauville-Fujiki relation $q_{BB}(D)>0$.
\end{rem}

\subsection{Bimeromorphic maps between irreducible symplectic manifolds}

A bimeromorphic map $f\colon M\rato M'$ between compact complex manifolds is an isomorphism in codimension $1$ if there exist dense open subsets $U\subset M$ and $U'\subset M'$ such that
\begin{enumerate}
\item $\codim(X\setminus U)\geq 2, \codim (X'\setminus U')\geq 2$;
\item $f$ induces an isomorphism $U\cong U'$.
\end{enumerate}
A \textbf{pseudo-automorphism} of a complex manifold $X$ is a bimeromorphic transformation which is an isomorphism in codimension $1$.

\begin{prop}[Proposition 21.6 and 25.14 in \cite{MR1963559}]
\label{pseudo-aut}
Let $f\colon X\rato X'$ be a bimeromorphic map between irreducible symplectic manifolds. Then $f$ is an isomorphism in codimension $1$ and induces a linear isomorphism $f^*\colon H^2(X',\Z)\xrightarrow{\sim} H^2(X,\Z)$ which preserves the Beauville-Bogomolov form.\\
In particular, the group of birational transformation of an irreducible symplectic manifold $X$ coincides with its group of pseudo-automorphisms and acts by isometries on $H^2(X,\Z)$.
\end{prop}

\subsection{Isometries of hyperbolic spaces}
\label{isometries hyperbolic}

Proposition \ref{signature BB} establishes a parallel between the dynamics of automorphisms of compact K\"ahler surfaces and that of bimeromorphic transformations of irreducible symplectic manifolds: in both cases the map on the manifold induces an isomorphism at the level of the integral cohomology group $H^2(X,\Z)$ preserving a non-degenerate quadratic form (the intersection form in the surface case and the Beauville-Bogomolov form in the irreducible symplectic one). By Hodge's index theorem, the intersection form on the Picard group of a surface $S$ has signature $(1, \rho(S)-1)$, which leads to a classification of automorphisms of surfaces as loxodromic, parabolic or elliptic depending on their action on the hyperbolic lattice $NS_\Z(S)$ (see \cite{MR3289919}).\\
Analogously if $X$ is an irreducible symplectic manifold, the restriction of the Beauville-Bogomolov form to $H^{1,1}(X,\R)$ has signature $(1,h^{1,1}(X)-1)$. Since $H^{1,1}(X,\R)$ is invariant by the action of a bimeromorphic transformation $f\colon X\rato X$, and since the two lines $\C\sigma$ and $\C\bar\sigma$ are also invariant (the action of $f^*$ being given by multiplication by a complex number of modulus $1$), we can also classify bimeromorphic transformations of irreducible symplectic manifolds depending on their action on $H^2(X,\Z)$ as follows.

\begin{dfn}
Let $f\colon X\rato X$ a bimeromorphic transformation of an irreducible symplectic manifold (respectively, an automorphism of a compact K\"ahler surface) and denote by $f_1^*$ the linear automorphism of $H^{1,1}(X,\R)$ induced by $f$. We say that $f$ is
\begin{itemize}
\item \emph{loxodromic} if $f_1^*$ admits an eigenvalue of modulus strictly greater than $1$ (or, equivalently, if $\lambda_1(f)>1$);
\item \emph{parabolic} if all the eigenvalues of $f_1^*$ have modulus $1$ and $||(f^n)_1^*||$ is not bounded as $n\to +\infty$;
\item \emph{elliptic} if $||(f^n)_1^*||$ is bounded as $n\to +\infty$.
\end{itemize}
\end{dfn}

In each of the cases above, simple linear algebra arguments allow to further describe the situation.\\
Denote by $\mathcal C_{\geq 0}\subset H^{1,1}(X,\R)$ (respectively $\mathcal C_0\subset H^{1,1}(X,\R)$) the positive (resp. null) cone for the Beauville-Bogomolov (repsectively, intersection) form $q$:
$$\mathcal C_{\geq 0}= \{\alpha\in H^{1,1}(X,\R) | q(\alpha)\geq 0\},$$
$$\mathcal C_{0}= \{\alpha\in H^{1,1}(X,\R) | q(\alpha)=0\}.$$
$\mathcal C_0$ is called the \emph{isotropic cone} for the Beauville-Bogomolov form. 

For a proof of the following result, see \cite{MR2533769} (for irreducible symplectic manifolds) and \cite{MR3289919} (for surfaces).

\begin{thm}
\label{characterization birational}
Let $f\colon X\rato X$ a bimeromorphic transformation of an irreducible symplectic manifold (respectively, an automorphism of a compact K\"ahler surface).
\begin{itemize}
\item If $f$ is loxodromic, then $f_1^*$ has exactly one eigenvalue with modulus $>1$ and exactly one eigenvalue with modulus $<1$; these eigenvalues are real, simple and they are the inverse of each another; their eigenspaces are contained in $\mathcal C_0$, they are the only $f_1^*$-invariant lines in $\mathcal C_{\geq 0}$ and they are not defined over $\Z$.
\item If $f$ is parabolic, then all eigenvalues of $f^*_1$ are roots of unity; the Jordan form of $f^*$ has exactly one non-trivial Jordan block, which is of dimension $3$ (in other words $||(f^n)_1^*||$ has quadratic growth); for every $\alpha\in H^{1,1}(X,\R)$, $(f^n)_1^*\alpha / n^2$ converges to a class contained in $\mathcal C_0$, which (for every $\alpha$ outside a proper subspace of $H^{1,1}(X,\R)$) spans the only $f_1^*$-invariant line of $\mathcal C_{\geq 0}$.
\item If $f$ is elliptic, then some iterate of $f_1^*$ is equal to the identity.
\end{itemize}
\end{thm}

\subsection{The parabolic case}
\label{parabolic case}

In the case of surfaces, an automorphism being of parabolic type has a clear geometric interpretation (see \cite{gizatullin1980rational}, \cite{grivaux2013parabolic}, or \cite{MR1867314} for the birational case).

\begin{thm}
\label{fibrations surfaces}
Let $S$ be a compact K\"ahler
 surface; an automorphism $f\colon S\to S$ is of parabolic type if and only if there exists an $f$-invariant fibration $\pi\colon S\to C$ onto a nonsingular compact curve $C$.

\end{thm}

We could expect the situation to be similar in the irreducible symplectic context; indeed, Hu, Keum and Zhang have proved a partial analogue to Theorem \ref{fibrations surfaces}, see \cite{MR3431659}:

\begin{thm}
\label{hu zhang}
Let $X$ be a $2n$-dimensional projective irreducible symplectic manifold of type $K3^{[n]}$ or of type generalized Kummer and let $f\in Bir(X)$ be a bimeromorphic transformation which is not elliptic; $f$ is parabolic if and only if it admits a rational Lagrangian invariant fibration $\pi\colon X\rato \Pj^n$ such that the induced transformation on $\Pj^n$ is biregular, i.e. there exists $g\in Aut(\Pj^n)$ such that $\pi\circ f=g\circ \pi$.
\end{thm}

The hard direction is to exhibit an invariant fibration for a parabolic transformation. The Main Theorem generalizes the converse, proving that the dynamics of a loxodromic transformation is too complicated to expect an invariant fibration.

\section{Proof of the main results}
\label{proof main results}

Throughout this section, $f\colon X\rato X$ denotes a loxodromic bimeromorphic transformation of an irreducible symplectic manifold $X$, $\pi\colon X\rato B$ a meromorphic invariant fibration onto a K\"ahler manifold $B$ such that $0<\dim B<\dim X$ and $g\colon B\rato B$ the induced transformation of the base.

\centerline{\xymatrix{
X \ar@{-->}[d]^\pi \ar@{-->}[r]^f & X \ar@{-->}[d]^\pi\\
B \ar@{-->}[r]^g & B
}}

The results in this Section are largely inspired by \cite{MR2400885}.

\subsection{Meromorphic fibrations on irreducible symplectic manifolds}
\label{fibrations HK}
We collect here some useful facts about the fibration $\pi$.

\begin{rem}
\label{projective base}
If $B$ is K\"ahler, then it is projective. Indeed, if $B$ wasn't projective, by Kodaira's projectivity criterion and Hodge decomposition
$$H^2(B,\C)=H^{2,0}(B)\oplus H^{1,1}(B)\oplus H^{0,2}(B),$$
we would have $H^{2,0}(B)\neq \{0\}$, meaning that $B$ carries a non-trivial holomorphic $2$-form $\sigma_B$. Since the indeterminacy locus of $\pi$ has codimension at least $2$, the pull-back $\pi^*\sigma_B$ can be extended to a global non-trivial $2$-form on $X$ which is not a multiple of $\sigma$, contradicting the hypothesis on $X$.
\end{rem}

Here we use the same conventions as in \cite{MR2400885}: let $\eta\colon \tilde X\to X$ be a resolution of the indeterminacy locus of $\pi$ (see \cite{MR0506253}), and let $\nu\colon \tilde X\to B$ be the induced holomorphic fibration, whose generic fibre is bimeromorphic to that of $\pi$. 

\centerline{\xymatrix{
\tilde X \ar[d]^\eta \ar[rd]^\nu&\\
X \ar@{-->}[r]^\pi & B
}}

The pull-back $\pi^*D$ of an effective divisor $D\in Div(B)$ is defined as
$$\pi^*D=\eta_*\nu^*D,$$
where $\eta_*$ is the pushforward as cycles. The pull-back induces linear morphisms $Pic(B)\to Pic(X)$ and $NS(B)\to NS(X)$, and is compatible with the pull-back of smooth forms defined in Section \ref{dyn deg}.


Now let $H\in Pic(B)$ be an ample class, and let $L=\pi^*H$. The pull-back of the complete linear system $|H|$ is a linear system $U\subset |L|$, whose associated meromorphic fibration is exactly $\pi$. In particular, $L$ has no fixed component, and by Remark \ref{non-fixed q non-negative} we have $q_{BB}(L)\geq 0$.

Let $NS(B)\subset H^{1,1}(B,\R)$ denote the Neron-Severi group with real coefficients of $B$. The following Lemma is essentially proven in \cite{MR2400885}.



\begin{lemma}
\label{pullback isotropic}

The restriction of the Beauville-Bogomolov form to the pull-back $\pi^*NS(B)$ is not identically zero if and only if the generic fibre of $\pi$ is of general type. If this is the case, then $X$ is projective.
\end{lemma}

\begin{proof}
Remark first that, since the generic fibre of $\nu$ is bimeromorphic to the generic fibre of $\pi$ and the Kodaira dimension is a bimeromorphic invariant, the generic fibre of $\pi$ is of general type if and only if the generic fibre of $\nu$ is.\\
As a second remark, by \cite{moishezon1967criterion} if there exists a big line bundle on a compact K\"ahler manifold $X$, then $X$ is projective.

Suppose that the generic fibre of $\pi$ is of general type. Let $H$ be an ample divisor on $B$ and let $L=\pi^*H$. By \cite{MR2400885}[Theorem 2.3] we have $\kappa(X,L)=\dim(B)+\kappa(F)$, where $F$ is the generic fibre of $\nu$; we conclude that $L$ is big (and in particular $X$ is projective). We can thus write $L=A+E$ for an ample divisor $A$ and an effective divisor $E$ on $X$. Now, if $q$ denotes the Beauville-Bogomolov form, we have
$$q(L)=q(L,A)+q(L,E)\geq q(L,A)= q(A,A)+q(A,E)\geq q(A,A)>0,$$
where the first and second inequalities are consequences of $L$ and $A$ being without fixed components and the last one follows directly from Remark \ref{non-fixed q non-negative}
. This proves the "if" direction.

Now assume that the restriction of $q_{BB}$ to $\pi^*NS(B)$ is not identically zero. Since ample classes generate $NS_\R(B)$, there exists an ample line bundle $H\in Pic(B)$ such that, denoting $L=\pi^*H$, $q(L)\neq 0$; furthermore, $L$ is without fixed components, so that $q(L)>0$ by Remark \ref{non-fixed q non-negative}. It follows by\cite{MR2050205}[Theorem 4.3.i] that $L$ is big (thus $X$ is projective), and so is $\eta^*L$ since $\eta$ is a birational morphism. Therefore, the restriction $\eta^*L|_F$ to a generic fibre of $\nu$ is also big (see \cite{MR2095471}[Corollary 2.2.11]).
Now we have
$$\eta^*L=\nu^*H+\sum a_iE_i \qquad \text{for some }a_i\geq 0,$$
where the sum runs over all the irreducible components of the exceptional divisor of $\eta$. The adjunction formula leads to
$$K_F=K_{\tilde X}|_F+\det N^*_{F/\tilde X}=K_{\tilde X}|_F=\sum e_iE_i|_F \qquad \text{for some }e_i> 0,$$
since the conormal bundle $N^*_{F/\tilde X}$ is trivial.\\
This implies that, for some $m>0$, the divisor $mK_F-\eta^*L|_F$ is effective because $\nu^*H|_F$ is trivial. Thus
$$\kappa(F)\geq \kappa(F,\eta^*L|_F)=\dim(F),$$
meaning that $F$ is of general type. This proves the "only if" direction.
\end{proof}

\begin{cor}
\label{pullback line}
If the generic fibre of $\nu$ is not of general type, then $\pi^*NS(B)\subset H^{1,1}(X,\R)$ is a line contained in the isotropic cone $\mathcal C_0$.
\end{cor}
\begin{proof}
By Lemma \ref{pullback isotropic}, $\pi^*NS(B)$ is contained in the isotropic cone. The pull-back $L$ of an ample line bundle on $B$ is effective and non-trivial, so that its numerical class is also non-trivial; thus $\pi^*NS(B)$ cannot be trivial. To conclude it suffices to remark that $\pi^*NS(B)$ is a linear subspace of $H^{1,1}(X,\R)$, and the only non-trivial subspaces contained in the isotropic cone are lines. 
\end{proof}

\subsection{Density of orbits}

The following theorem was proven in \cite{MR2400885}.

\begin{thm}
\label{orbit fibration}
Let $X$ be a compact K\"ahler manifold and let $f\colon X\rato X$ be a dominant meromorphic endomorphism. Then there exists a dominant meromorphic map $\pi\colon X\rato B$ onto a compact K\"ahler manifold $B$ such that
\begin{enumerate}
\item $\pi \circ f=\pi$;
\item the general fibre $X_b$ of $\pi$ is the Zariski closure of the orbit by $f$ of a generic point of $X_b$.
\end{enumerate}
\end{thm}

\begin{lemma}
\label{pullback pseudo-aut}
Let $\phi\colon X\rato Y$, $\psi\colon Y\rato Z$ be meromorphic maps between compact complex manifolds. If $\phi$ is an isomorphism in codimension $1$, then for all $D\in Div(Z)$
$$(\psi\circ \phi)^*D=\phi^*\psi^*D.$$
\end{lemma}
\begin{proof}
Let $U\subset X$, $V\subset Y$ two open sets such that $\phi$ induces an isomorphism $U\cong V$ and such that $\codim (X\setminus U)\geq 2,\codim(Y\setminus V)\geq 2$. It is easy to see that, for every effective divisor $D_Y\in Div(Y)$, we have an equality $\phi^*D_Y=\overline{\phi|_U^*(D_Y\cap V)}$; therefore the equality is true for every divisor in $Div(Y)$.\\
Up to shrinking $V$ to some other open subset whose complement has codimension at least $2$, we can suppose that $\psi$ is regular on $V$; therefore the composition $\psi\circ \phi$ is regular on $U$, and since the complement of $U$ has codimension $\geq 2$ in $X$, for all $D\in Div(Z)$ we have
$$(\psi\circ \phi)^*D=\overline{(\psi\circ \phi)|_U^*D}=\overline{\phi|_U^*(\psi|_V^*D)}=\overline{\phi|_U^*(\psi^*D\cap V)}=\phi^*\psi^*D,$$
where the third equality follows again from the fact that the complement of $V$ has codimension at least $2$ in $Y$. This proves the claim.
\end{proof}

Let us prove point $(3)$ of the Main Theorem.

\begin{lemma}
\label{dense orbits}
Let $f\colon X\rato X$ be a bimeromorphic loxodromic transformation of an irreducible symplectic manifold. Then the generic orbit of $f$ is Zariski-dense.
\end{lemma}
\begin{proof}
If the claim were false, then by Theorem \ref{orbit fibration} we could construct a commutative diagram\\
\centerline{\xymatrix{
X \ar@{-->}[d]^\pi \ar@{-->}[r]^f & X \ar@{-->}[d]^\pi\\
B \ar[r]^{id} & B
}}
where $\pi$ is a meromorphic map whose general fibre $X_b$ coincides with the Zariski-closure of the $f$-orbit of a generic point of $X_b$. Remark \ref{projective base} applies in the case where the fibres are not connected; therefore the base $B$ is projective.\\
Now, remark that $f^*$ acts as the identity on the space $\pi^*NS(B)\subset NS(X)$, which is defined on $\Q$: indeed, for $v\in NS(B)$, we have
$$f^*\pi^*v=(\pi\circ f)^*v=(id_B \circ \pi)^* v =\pi^*v,$$
where the first equality follows from Lemma \ref{pullback pseudo-aut}; by Theorem \ref{characterization birational}, the Beauville-Bogomolov form is negative definite on $\pi^*NS(B)$.\\
Now, let $H\in Pic(B)$ be an ample line bundle and let $L=\pi^*H$. We have seen in \ref{fibrations HK} (again the hypothesis on fibres being connected was irrelevant) that $L$ is a numerically non-trivial line bundle such that $q_{BB}(L)\geq 0$, contradiction. This proves the claim.

\end{proof}

\subsection{The key lemma}

The following key lemma, together with the Proposition in Section \ref{intro}, implies the Main Theorem.

\begin{lemma}[Key lemma]
\label{key lemma}
Let $X$ be an irreducible symplectic manifold, $f\colon X\rato X$ a loxodromic bimeromorphic transformation and $\pi\colon X \rato B$ a meromorphic $f$-invariant fibration onto a compact K\"ahler manifold. Then $X$ is projective and the generic fibre of $\pi$ is of general type.
\end{lemma}
\begin{proof}
Let $g\colon B\rato B$ be a bimeromorphic transformation such that $g\circ \pi=\pi\circ f$.

Let us define
$$V:=Span\left\{(h \circ \pi)^*NS_\R(B)| h \colon B\rato  B  \text{ birational transformation}\right\}\subset NS_\R(X).$$
The linear subspace $V$ is clearly defined over $\Q$. Since the pull-back by $\pi$ of an ample class is numerically non-trivial, we also have $V\neq\{0\}$.\\
Furthermore, $V$ is $f^*$-invariant: if $v=(h\circ \pi)^*w$ for some $w\in NS(B)$ and for some birational transformation $h\colon B\rato B$, then
$$f^*v=f^*(h\circ \pi)^* w=(h\circ \pi\circ f)^*w=(h\circ g\circ \pi)^*w=(\tilde h \circ \pi)^*w,$$
where $\tilde h =h\circ g\colon B\rato B$ is a birational transformation and the second equality follows from Lemma \ref{pullback pseudo-aut}.

Now suppose that the generic fibre of $\pi$ is not of general type; we are first going to show that $V$ is contained in the isotropic cone $\mathcal C_0=\{v\in H^{1,1}(X,\R)| q_{BB}(v)=0\}$. The generic fibre of the meromorphic fibration $h\circ \pi$ is bimeromorphic to that of $\pi$. By Lemma \ref{pullback isotropic} we know that $(h\circ \pi)^*NS_\R(B)$ is contained in the isotropic cone for all birational transformations $h\colon B\rato B$. We just need to show that for all birational transformations of $B$ onto itself $h_i,h_j$ and for all $w_i, w_j\in NS_\R(B)$ we have
$$q_{BB}((h_i\circ \pi)^*w_i, (h_j\circ \pi)^*w_j)=0.$$
Let $h=h_j\circ h_i^{-1}$, and let $\rho \colon \tilde B\to B$ be a resolution of the indeterminacy locus of $h$; denote by $\tilde h \colon \tilde B \to B$ the induced holomorphic transformation, and let $\tilde \pi=\rho^{-1} \circ h_i \circ \pi\colon X\rato \tilde B$; $\tilde \pi$ is a meromorphic fibration onto the birational model $\tilde B$, whose generic fibre is bimeromorphic to that of $g$. Finally, let $\eta\colon \tilde X\to X$ be a resolution of singularities of $\tilde \pi$ and let $\nu\colon \tilde X\to B$ be the induced holomorphic map.

\centerline{
\xymatrix{
\tilde X \ar[dd]^\eta \ar[rd]^\nu&&\\
&\tilde B \ar[d]^\rho \ar[dr]^{\tilde h}&\\
X \ar@{-->}[r]^{h_i\circ \pi} \ar@{-->}[ru]^{\tilde\pi}& B \ar@{-->}[r]^h &B
}}

Now it is clear that $\eta\colon \tilde X\to X$ is a resolution of singularities of both $h_i\circ \pi$ and $h_j\circ \pi=h\circ h_i\circ \pi$. Therefore
$$(h_i\circ \pi)^*w_i=\eta_*\nu^*\rho^*w_i=\tilde \pi^*\rho^*w_i\in \tilde \pi^*NS(\tilde B)$$
and
$$(h_j\circ \pi)^*w_j=\eta_*\nu^*\tilde h^*w_j=\tilde\pi^* \tilde h^*w_j\in \tilde \pi^* NS(\tilde B).$$
Since the fibres of $\tilde \pi$ are not of general type, it suffices to apply Lemma \ref{pullback isotropic} to the fibration $\tilde \pi\colon X\rato \tilde B$ to conclude that $q_{BB}((h_i\circ \pi)^*w_i, (h_j\circ \pi)^*w_j)=0$. This proves that $V$ is contained in the isotropic cone.

Now the only non trivial vector subspaces of $NS_\R(X)$ contained in the isotropic cone are lines; by Theorem \ref{characterization birational}, $V$ is then an $f^*$-invariant line contained in the isotropic cone and not defined over $\Q$. But this contradicts the definition of $V$. We have thus proved that the generic fibre of $\pi$ is of general type.

In order to prove that $X$ is projective it suffices to apply the last part of Lemma \ref{pullback isotropic}.


\end{proof}

By \cite[Corollary 14.3]{MR0506253}) we know that the group of birational transformations of a variety of general type is finite. Therefore, we expect the dynamics of $f$ on the fibres to be simple.

\subsection{Relative Iitaka fibration}
\label{relative fibration}

Before giving the proof of the Main Theorem, we are going to recall the basic results about the relative Iitaka fibration. We will follow the approach of \cite{tsuji2010global} with some elements from \cite{MR0506253}. See also \cite{MR3075000}, \cite{MR0217084}.

Let $X$ be a smooth projective variety, and suppose that some multiple of $K_X$ has some non trivial section. Recall that, for $m>0$ divisible enough, the rational map
\begin{align*}
\phi_{|mK_X|} \colon X & \rato \Pj H^0(X,mK_X)^\vee\\
p & \mapsto \{s\in \Pj H^0(X,mK_X) | s(p)=0\}
\end{align*}
has connected fibres. Moreover the rational map $\phi_{|mK_X|}$ eventually stabilize to a rational fibration that we call \emph{canonical fibration} of $X$.

\begin{rem}
\label{induced linear}
If $f\colon X\rato X$ is a bimeromorphic transformation of $X$, the pull-back of forms induces a linear automorphism $f^*\colon H^0(X,mK_X)\to H^0(X,mK_X)$. For example, for $m=1$ a section $\sigma \in H^0(X,K_X)$ is a holomorphic $d$-form ($d=\dim X$); $f$ is defined on an open set $U\subset X$ such that $X\setminus U$ has codimension at least $2$. Therefore by Hartogs theorem the pull-back $f|_U^*\sigma$ can be extended to $X$. It is easy to see that the construction is invertible and induces a linear automorphism of $\Pj H^0(X,mK_X)^\vee$ which commutes with the Iitaka fibration:

\centerline{
\xymatrix{
 X \ar@{-->}[r]^f \ar@{-->}[d]^{\phi_{|mK_X|}}& X \ar@{-->}[d]^{\phi_{|mK_X|}}\\
\Pj H^0(X,mK_X)^\vee \ar[r]^{\tilde f}& \Pj H^0(X,mK_X)^\vee
}}
\end{rem}

The above construction can be generalized to the relative setting: let $\pi\colon X\to B$ be a regular fibration onto a smooth projective variety $B$, and let $K_{X/B}=K_X\otimes \pi^*K_B^{-1}$ be the relative canonical bundle.\\
For some fixed positive integer $m>0$ (divisible enough), let $\mathcal S=\pi_*(mK_{X/B})^{\vee}$. $\mathcal S$ is a coherent sheaf over $B$; therefore one can construct (generalizing the construction of the projective bundle associated to a vector bundle, see \cite{MR0506253} for details) the algebraic projective fibre space
$$\eta \colon \Proj(\mathcal S)\to B$$
associated to $\mathcal S$, which is a projective scheme (a priori neither reduced nor irreducible) $Y$ over $B$. Its generic geometric fibre $Y_b$ over a generic point $b\in B$ is canonically isomorphic to $\Pj H^0(X_b, mK_{X_b})^\vee $. The Iitaka morphisms $\phi_b\colon X_b\rato \Pj H^0(X_b, mK_{X_b})^\vee$ induce a rational map $\phi\colon X\rato Y$ over $B$.\\
The \emph{relative canonical fibration} of $X$ with respect to $\pi$ is
\begin{align*}
\phi\colon X & \rato Y\\
x\in X_b &\mapsto \left[\{s\in H^0(X_b; mK_{X_b})| s(x)=0\}\right]\in Y_b.
\end{align*}
It can be shown that, for $m$ divisible enough:
\begin{itemize}
\item $\phi$ stabilizes to a certain rational fibration;
\item the image by $\phi$ of the generic fibre $X_b=\pi^{-1}(b)$ of $\pi$ is contained inside the fibre $\eta^{-1}(b)$ of the natural projection $\eta\colon Y\to B$;
\item the restriction of $\phi$ to a generic fibre $X_b$ is birationally equivalent to the canonical fibration of $X_b$.
\end{itemize}

\begin{rem}
\label{induced linear relative}
The construction in Remark \ref{induced linear} can also be generalized to the relative setting: let $f\colon X\rato X$ and $g\colon B\rato B$ be birational transformations such that $\pi\circ f=g\circ \pi$. For a generic $b\in B$ define
\begin{align*}
\tilde f|_{Y_b}\colon \Pj H^0(X_{b},mK_{X_{b}})^\vee &\rato \Pj H^0(X_{g(b)},mK_{X_{g(b)}})^\vee\\
[s^*]&\mapsto \left\{  [s]\in  \Pj H^0(X_{g(b)},mK_{X_{g(b)}})| s^*(f^*s)=0  \right\}.
\end{align*}
These are well defined linear automorphisms because, for a fibre $X_b$ of $\pi$ not contained in the indeterminacy locus of $f$, the restriction $f\colon X_b \rato X_{g(b)}$ is a birational map, and thus induces a linear isomorphism 
$$f^*\colon H^0(X_{g(b)},mK_{X_{g(b)}})\to H^0(X_b,mK_{X_b}).$$
Furthermore the $\tilde f_{X_b}$ can be glued to a birational transformation $\tilde f\colon Y\rato Y$ such that $\eta\circ \tilde f=g\circ \eta$.
\end{rem}

Now suppose the generic fibre of $\pi$ is of general type. Since the restriction of $\phi$ to a generic fibre of $g$ is birational onto its image and the images of fibres are disjoint, $\phi$ itself must be birational onto its image; denote by $Z$ the closure of the image of $\phi$ and let $f_Z=\phi\circ f\circ \phi^{-1} \colon Z\rato Z$ be the birational transformation induced by $f$. \\
By the above Remark, $f_Z$ is the restriction of the birational transformation $\tilde f\colon Y \rato Y$. In particular $f_Z$ induces an isomorphism between generic fibres of $\eta|_Z$.

\centerline{\xymatrix{
X \ar@{-->}[d]^\pi \ar@{-->}[r]^\phi  \ar@(ul,ur)[]|{f} & Z \ar@(ul,ur)[]|{f_Z} \ar@{^{(}->}[r] \ar@{-->}[ld]& Y \ar@(ul,ur)[]|{\tilde f} \ar@{-->}[lld]^{\eta}\\
B \ar@(dl,ul)[]|{g}
}}
\vspace{0.3cm}

\subsection{Proof of the Main Theorem}

\begin{lemma}
\label{dense orbit isotrivial}
Let $X,B$ be projective manifolds, $f\colon X\rato X$ and $g\colon B\rato B$ birational transformations and $\pi\colon X\to B$ a fibration such that $\pi\circ f=g\circ \pi$.\\
\centerline{\xymatrix{
X \ar[d]^\pi \ar@{-->}[r]^f & X \ar[d]^\pi\\
B \ar@{-->}[r]^g & B
}}\\
Assume that the generic fibre of $\pi$ is of general type and that
 the generic orbit of $g$ is Zariski-dense. Then all the fibres over a non-empty Zariski open subset of $B$ are isomorphic.
\end{lemma}
\begin{proof}

Denote as before
$$\phi\colon X \rato Y$$
the relative Iitaka fibration. We are going to identify $X$ with its birational model $\overline{\phi(X)}$.\\
Let $F=\pi^{-1}(b_0)$ be the fibre of $\pi$ over a point $b_0$ whose orbit is Zariski-dense in $B$, and let 
$$\mathfrak{I}:=\mathfrak{Isom}_B(X,F\times B)$$
be the $B$-scheme of isomorphisms over $B$ between $X$ and $F\times B$; the fibre $\mathfrak I_b$
parametrizes the isomorphisms $X_b\cong F$. We can realize $\mathfrak I$ as an open subset of the Hilbert scheme $\mathfrak{Hilb}_B(X\times_B (B\times F))$ by identifying a morphism $X_b\to F$ with its graph in $X_b\times F$. Therefore, 
$$\mathfrak I=\coprod _{P\in \Q[\lambda]} \mathfrak I^P,$$
where the fibre $\mathfrak I^P_b$ is the (a priori non irreducible and non reduced) quasi-projective scheme of (graphs of) isomorphisms $X_b\xrightarrow{\sim} F$
having fixed Hilbert polynomial $P(\lambda)$; such polynomials are calculated with respect to the restriction to the fibre $X_b\times F$ of a fixed line bundle $L$ on $X\times_B (B\times F)$ relatively very ample over $B$. We shall fix 
$$L=H_Y|_{X} \boxtimes_B H_F,$$
where $H_Y$ is a very ample line bundle on $Y$ and $H_F$ is a very ample line bundle on $F$.\\
Now, the pull-back of forms by $f$ induces a linear isomorphism
$$\tilde f_b\colon \Pj H^0(X_b, mK_{X_b})^\vee \xrightarrow{\sim} \Pj H^0(X_{g(b)}, mK_{X_{g(b)}})^\vee $$
between fibres of $\eta\colon Y\xrightarrow{\sim} B$, which restricts to an isomorphism $X_b\to X_{g(b)}$; under the canonical identification of fibres of $\eta$ with $\Pj^N$, $H_Y|_{Y_b}\cong \mathcal O_{\Pj^N}(d)$ (meaning that the section $H_Y|_{Y_b}$ has degree $d$)
 for some $d>0$ independent of the fibre. Under the identification, the action of $\tilde f_b$ is linear, so that $\tilde f_b^*(H_Y|_{Y_{g(b)}})$ also has degree $d$ on $\Pj^N$. In particular we have
$$\tilde f_b^*(H_Y|_{X_{g(b)}})=H_Y|_{X_b}.$$
Now take any isomorphism $X_{b_0}\xrightarrow{\sim} F$, which we can identify with its graph $\Gamma\subset X_{b_0}\times F$; the image of $\Gamma$ by the isomorphism $\tilde f_{b_0}\times id_F\colon X_{b_0}\times F\xrightarrow{\sim} X_{g(b_0)}\times F$ is the graph $\Gamma'$ of an isomorphism $X_{g(b_0)}\xrightarrow{\sim} F$. Furthermore, since $(\tilde f_{b_0}\times id_F)^*(L|_{X_{g(b_0)}\times F})=L|_{X_{b_0}\times F}$, $\Gamma'$ has the same Hilbert polynomial as $\Gamma$. Iterating this reasoning we find that for some $P\in \Q[\lambda]$ the image of the natural morphism $\psi\colon \mathfrak I^P\to B$ is Zariski-dense.\\
By Chevalley's theorem (\cite[Theorem 3.16]{MR1416564}) we also know that $\psi(\mathfrak I^P)$ is constructible; since every constructible Zariski-dense subset of an irreducible scheme contains a dense open set \cite[Proof of Theorem 3.16]{MR1416564}, we have $X_b\cong F$ for all $b$ in an open dense subset of $B$. This concludes the proof.
\end{proof}

\begin{proof}[Proof of the Proposition in Section \ref{intro}]

By Lemma \ref{dense orbit isotrivial}, all the fibres over a dense open subset $U\subset B$ are isomorphic, which shows $(1)$.  By \cite[Proposition 2.6.10]{MR2247603}, there exists an étale cover $\epsilon \colon U'\to U$ such that the induced fibration $X'_{U'}:=X\times_U U'$ is trivial: 
$$X'_{U'}\cong U'\times F.$$
This shows $(2)$.

Now suppose that the generic fibre $F$ is of general type. This implies that the group $G:=\Aut(F)$ is finite; therefore, for any $x\in F$, we can define the subvariety 
$$W^x:=U'\times G\cdot x\subset X'_{U'}.$$
We are going to show that the image of $W^x$ by the cover $\epsilon_X\colon X'_{U'}\to X_U$ is $f$-invariant.

Remark that the fibration $\pi_U\colon X_U\to U$ is locally trivial in the euclidean topology. Let $\{U_i\}_{i\in I}$ be a covering of $U$ by euclidean open subsets such that the restriction of the fibration to each $X_{U_i}$ is trivial: there exist biholomorphisms $X_{U_i}\cong U_i\times F$. Then the subvarieties 
$$V_i^x:=U_i\times  G\cdot x \subset X_{U_i}$$
patch together to algebraic subvarieties $V^x$ of $X_U$ which are exactly the images of the $W^x$.\\
Now we will prove that the varieties $V^x$ are $f$-invariant. Let $p\in X_U$ be a point where $f$ is defined and such that $g$ is defined on $\pi(p)$, and let $i\in I$ be such that $p\in X_{U_i}$; up to shrinking $U_i$, we can suppose that $g(U_i)\subset U_j$. By an identification $X_{U_i}\cong U_i\times F, X_{U_j}\cong U_j\times F$, we can write $f(x,y)=(g(x),h(x,y))$; here, for all $x$ in on open dense subset of $U_i$, the continuous map $b\mapsto h(x,\bullet)\in \Bir(F)$ is well defined. Since $\Bir(F)$ is a finite (hence discrete) group, $h$ doesn't depend on $x$, which shows that all the varieties $V^x$ are $f$-invariant. 

Now remark that the varieties $\epsilon_X(U'\times G\cdot \{x\})$ are the disjoint union of varieties of type $\epsilon_X(U'\times \{y\})$; since the first are $f$-invariant, the latter must be $f$-periodic, which concludes the proof.

\end{proof}

\begin{proof}[Proof of the Main Theorem, point (1)]
Let $X$ be an irreducible symplectic manifold, $f\colon X\rato X$ a birational loxodromic transformation, and suppose by contradiction that $f$ is imprimitive: there exist thus a meromorphic fibration $\pi\colon X\rato B$ and a bimeromorphic transformation $g\colon B\rato B$ such that $\pi\circ f=g\circ \pi$.\\
By Lemma \ref{key lemma}, $X$ is projective and the generic fibre of $\pi$ is of general type. However, we also know by Lemma \ref{dense orbits} that the generic orbit of $f$ is Zariski-dense; therefore, by the Proposition in Section \ref{intro}, the generic fibre of $\pi$ cannot be of general type, a contradiction.
\end{proof}

\section{Invariant subvarieties}
\label{sec: invariant subvarieties}


Let $X$ be a compact complex manifold. If $f\colon X\to X$ is an automorphism, we say that a subvariety $W\subset X$ is \emph{invariant} if $f(W)=W$, or, equivalently, if $f^{-1}(W)=W$. We say that $W\subset X$ is \emph{periodic} if it is invariant for some positive iterate $f^n$ of $f$.\\
Now let $f\colon X\rato X$ be a pseudo-automorphism of $X$ (i.e. a bimeromorphic transformation which is an isomorphism in codimension $1$). We say that a hypersurface $W\subset X$ is invariant if the strict transform $f^*W$ of $W$ is equal to $W$ (as a set); since $f$ and $f^{-1}$ don't contract any hypersurface, this is equivalent to $f(W)=W$ (here $f(W)$ denotes the analytic closure of $f|_U(W\cap U)$, where $U\subset X$ is the maximal open set where $f$ is well defined). We say that a hypersurface is periodic if it is invariant for some positive iterate of $f$.


The following Theorem is a special case of \cite{MR2727612}[Theorem B].

\begin{thm}
\label{invariant hypersurfaces fibration}
Let $f\colon X\rato X$ be a pseudo-automorphism of a compact complex manifold $X$. If $f$ admits at least $\dim(X)+b_2(X)-1$ invariant hypersurfaces, then it preserves a non-constant meromorphic function.
\end{thm}

\begin{proof}[Proof of the Main Theorem, point (2)]
Let $f\colon X\rato X$ be a loxodromic bimeromorphic transformation of an irreducible symplectic manifold $X$ (which is a pseudo-automorphism by \ref{pseudo-aut}).\\
Suppose that $f$ admits more than $\dim(X)+b_2(X)-2$ periodic hypersurfaces; then some iterate of $f$ satisfies the hypothesis of Theorem \ref{invariant hypersurfaces fibration}. Therefore $f^n$ preserves a non-constant meromorphic function $\pi\colon X\to \Pj^1$, and, up to considering the Stein factorization of $\pi$, we can assume that $\pi$ is an $f^n$-invariant fibration onto a curve. This contradicts point $(1)$ of the Main Theorem.
\end{proof}

The following example shows that we cannot hope to obtain an analogue of point $(2)$ of the Main Theorem for higher codimensional subvarieties.

\begin{ex}
Let $f\colon S\to S$ be a loxodromic automorphism of a $K3$ surface $S$, and let $X=Hilb^n(S)$. Then $X$ is an irreducible symplectic manifold and $f$ induces a loxodromic automorphism $f_n$ of $X$. By point $(2)$ of the Main Theorem, $f_n$ admits only a finite number of invariant hypersurfaces. However $f$ admits infinitely many periodic points (\cite{MR1864630},\cite{MR3289919}); if $x$ is a periodic point in $S$, then (the image in $X$ of) $\{x\}^p\times S^{n-p}$ is a periodic subvariety of codimension $2p$.
\end{ex}

Thus we have showed the following Proposition.

\begin{prop}
For all integers $0<p\leq n$, there exist a $2n$-dimensional projective irreducible symplectic manifold $X$ and a loxodromic automorphism $f\colon X\to X$ admitting infinitely many periodic subvarieties of codimension $2p$.
\end{prop}

\section{Appendix: An alternative approach to the Main Theorem}
\label{alternative approach}

In this section we describe a different approach to the proof of the Main Theorem which doesn't require the Proposition in Section \ref{intro}. The result we obtain is actually slightly weaker than the Main Theorem; however this approach allows to prove point $(2)$ and $(3)$, as well as point $(1)$ for automorphisms.

We have already seen in Proposition \ref{characterization birational} that the first dynamical degree of a bimeromorphic transformation $f\colon X\rato X$ is either $1$ or an algebraic integer $\lambda$ whose conjugates over $\Q$ are $\lambda^{-1}$ and some complex numbers of modulus $1$ (so that $\lambda$ is a quadratic or a Salem number). In the case of automorphisms, the following Proposition from Verbitsky \cite{MR1406664} allows to completely describe all the other dynamical degrees as well.

\begin{prop}
\label{verbitsky}
Let $X$ be an irreducible symplectic manifold of dimension $2n$ and let $SH2(X,\C)\subset H^*(X,\C)$ be the subalgebra generated by $H^2(X,\C)$. Then we have an isomorphism
$$SH^2(X,\C)=\Sym ^*H^2(X,\C)/\langle \alpha^{n+1}|q_{BB}(\alpha)=0 \rangle$$.
\end{prop}

The following Corollary is due to Oguiso \cite{MR2533769}.

\begin{cor}
\label{remark oguiso}
Let $f\colon X\to X$ be an automorphism of an irreducible symplectic manifold of dimension $2n$. Then for $p=0,1,\ldots, n$
$$\lambda_p(f)=\lambda_{2n-p}(f)=\lambda_1(f)^p.$$
\end{cor}
\begin{proof}
By Proposition \ref{verbitsky} the cup-product induces an injection
$$\Sym^p H^2(X, \C)\hookrightarrow H^{2p}(X,\C)$$
for $p=1,\ldots, n$.\\
Let $v_1\in H^2(X,\C)$ be an eigenvector for the eigenvalue $\lambda=\lambda_1(f)$. Then $v_p:=v_1^p\in H^{2p}(X,\C)$ is a non-zero class for $p=1,\ldots, n$ and $f^*v_p=(f^*v_1)^p=\lambda^p v_p$. This implies that $\lambda_p(f)\geq \lambda_1(f)^p$, and we must have equality by log-concavity (Proposition \ref{log-concavity}). This proves the result for $p=0,1,\ldots, n$.\\
Now by Remark \ref{inverse degrees} we have $\lambda_{2n-p}(f)=\lambda_p(f^{-1})$. Applying what we have just proved to $f^{-1}$ we obtain
$$\lambda_{2n-1}(f)=\lambda_1(f^{-1})=\lambda_n(f^{-1})^{1/n}=\lambda_n(f^{-1})^{1/n}=\lambda_1(f)$$
and thus, for $p=0, \ldots ,n$,
$$\lambda_{2n-p}(f)=\lambda_{p}(f^{-1})=\lambda_1(f^{-1})^p=\lambda_1(f)^p,$$
which concludes the proof.
\end{proof}

\begin{lemma}

\label{trivial relative dyn}
Let $X$ be a smooth projective variety, $f\colon X \rato X$ a birational transformation of $X$, $\pi\colon X\rato B$ a rational $f$-invariant fibration onto a smooth projective variety $B$. If the generic fibre of $\pi$ is of general type, then all the relative dynamical degrees $\lambda_p(f|\pi)$ are equal to $1$ (for $p=0, \ldots , \dim(X)-\dim(B)$).

\end{lemma}

\begin{proof}
Since the Kodaira dimension and the relative dynamical degrees are bimeromorphic invariants (Remark \ref{bimeromorphic invariant}), up to considering a resolution of the indeterminacy locus of $\pi$, we can suppose that $\pi$ is regular.\\
Let
\centerline{\xymatrix{
\phi\colon X \ar[d]^\pi \ar[r] & Y:=\Proj(\pi_*K_{X/Y}^{\otimes m}) \ar[ld]^\eta\\
B }}
be the Iitaka fibration. Since $\phi$ is birational onto its image, denoting $Z\subset Y$ the closure of $\phi(X)$, the claim is equivalent to $\lambda_p(f_Z|\eta_Z)=1$, where $\eta_Z$ denotes the restriction of $\eta$ to $Z$ and $f_Z=\phi\circ f\circ \phi^{-1} \colon Z\rato Z$.\\
The construction of Remark \ref{induced linear relative} provides a birational transformation $\tilde f\colon Y\rato Y$ extending $f_Z$.

\centerline{\xymatrix{
X \ar[d]^\pi \ar@{-->}[r]^\phi  \ar@(ul,ur)[]|{f} & Z \ar@(ul,ur)[]|{f_Z} \ar@{^{(}->}[r] \ar[ld]& Y \ar@(ul,ur)[]|{\tilde f} \ar[lld]^{\eta}\\
B \ar@(dl,ul)[]|{g}
}}
\vspace{0.3cm}

Now we will prove that if $\lambda_p(\tilde f|\eta)=1$ then $\lambda_p(f_Z|\eta_Z)=1$. Let $H_Y\in Pic(Y)$ and $H_B\in Pic(B)$ be ample classes; therefore $H_Y|_Z$ is an ample class on $Z$. The map
\begin{align*}
H^{2n,2n}(Y,\R)&\to \R\\
\alpha &\mapsto \int_Y \alpha \wedge c_1(H_Y)^{\dim(Y)-\dim(X)}=\alpha \cdot H_Y^{\dim(Y)-\dim(X)} 
\end{align*}
is linear and strictly positive (except on $0$) on the closed positive cone $\mathcal K_{2n}\subset H^{2n,2n}(Y,\R)$. Since $\alpha\mapsto \alpha\cdot [Z]$ is linear too, we can define
$$M=\max_{\alpha\in  \mathcal K_{2n}\setminus \{0\}}\frac {\alpha\cdot [Z]}{\alpha \cdot H_Y^{\dim(Y)-\dim(X)} }\geq 0.$$
Now
$$\lambda_p(f_Z|\eta_Z)=\lim_{n\to + \infty} \left(  (\tilde f^n)^*H_Y^p\cdot \eta^* H_B^{\dim(B)}\cdot H_Y^{2n-p-\dim(B)}\cdot [Z] \right)^{\frac 1n}\leq$$
$$\lim_{n\to + \infty} \left(  M(\tilde f^n)^*H_Y^p\cdot \eta^* H_B^{\dim(B)}\cdot H_Y^{\dim(Y)-p-\dim(B)} \right)^{\frac 1n}=$$
$$\lim_{n\to + \infty} \left(  (\tilde f^n)^*H_Y^p\cdot \eta^* H_B^{\dim(B)}\cdot H_Y^{\dim(Y)-p-\dim(B)} \right)^{\frac 1n}=\lambda_p(\tilde f|\eta)=1,$$
and since all relative dynamical degrees are $\geq 1$ (Proposition \ref{log-concavity}) we have $\lambda_p(f_Z|\eta_Z)=1$.

Now all is left to prove is that $\lambda_p(\tilde f|\eta)=1$. There exists $k>0$ such that $\eta^*H_B^{\dim(B)}\equiv_{num}k[F]$, where $[F]$ is the numerical class of a fibre $F$ of $\eta$. We have
$$\lambda_p(\tilde f|\eta)=\lim_{n\to + \infty} \left(  (\tilde f^n)^*H_Y^p\cdot \eta^* H_B^{\dim(B)}\cdot H_Y^{\dim(Y)-p-\dim(B)} \right)^{\frac 1n}=$$
$$\lim_{n\to + \infty} \left(  (\tilde f^n)^*H_Y^p\cdot k[F]\cdot H_Y^{\dim(Y)-p-\dim(B)} \right)^{\frac 1n}=$$
$$\lim_{n\to + \infty} \left(  \left((\tilde f^n)^*H_Y\right)|_F^p\cdot H_Y|_F^{\dim(Y)-p-\dim(B)} \right)^{\frac 1n}.$$
For each fibre we have a canonical identification $F\cong \Pj^N$, and by this identification $H_Y|_F\cong \mathcal O_{\Pj^N}(d)$, meaning that the hyperplane section $H_Y|_F$ is defined by an equation of degree $d$. Under the identification, the action of $\tilde f$ from one fibre to another is linear, so that $\left((\tilde f^n)^*H_Y\right)|_F$ is also defined by an equation of degree $d$ on $\Pj^N$. This means that
$$\lambda_p(\tilde f|\eta)=\lim_{n\to +\infty}(d^{\dim(F)})^{\frac 1n}=1$$
as we wanted to show. This concludes the proof.

\end{proof}

The following Proposition is a weaker version of point $(1)$ of the Main Theorem.

\begin{prop}
\label{inequality dim B}
Let $f\colon X\rato X$ be a loxodromic transformation of an irreducible symplectic manifold $X$ of dimension $2n$, and let
$$1=\lambda_0(f)<\cdots < \lambda_{p_0}(f)=\cdots =\lambda_{p_0+k}(f)>\cdots>\lambda_{2n}(f)=1$$
be its dynamical degrees.\\
If $\pi\colon B\rato B$ is an $f$-invariant meromorphic fibration, then $\dim(B)\geq 2n-k$. In particular, if $f$ is an automorphism (or, more generally, if all the consecutive dynamical degrees of $f$ are distinct), then it is primitive.
\end{prop}

\begin{proof}
Let $g\colon B\rato B$ be a birational transformation such that $g\circ \pi=\pi\circ f$.\\
\centerline{\xymatrix{
X \ar@{-->}[d]^\pi \ar@{-->}[r]^f & X \ar@{-->}[d]^\pi\\
B \ar@{-->}[r]^g & B
}}
We know by Lemma \ref{key lemma} that the generic fibre of $\pi$ is of general type; by Lemma \ref{trivial relative dyn} this implies that all the relative dynamical degrees $\lambda_p(f|\pi)$ are equal to $1$. By Theorem \ref{dinh nguyen} we then have
$$\lambda_p(f)=\max_{p-\dim(F)\leq q\leq p}\lambda_q(g),$$
where $\dim(F)=\dim(X)-\dim(B)$ is the dimension of a generic fibre.\\
Let $q\in \{0, 1, \ldots, \dim(B)\}$ be such that $\lambda_q(g)$ is maximal. Then 
$$\lambda_q(f)=\lambda_{q+1}(f)=\cdots= \lambda_{q+\dim(F)}(f)=\lambda_q(g),$$ 
meaning that $k\geq \dim(F)=2n-\dim(B)$. This concludes the proof.

\end{proof}

\begin{rem}
Since in the Theorem we have $k\leq 2n-1$, the base of an invariant fibration cannot be a curve. Therefore Proposition \ref{inequality dim B} implies point $(2)$ of the Main Theorem.
\end{rem}

\bibliography{references.bbl}{}
\bibliographystyle{plain}
\end{document}